\documentclass[12pt]{amsart}
\usepackage{amsmath}
\usepackage{amssymb}
\usepackage{amsmath,amssymb,amsthm,amscd}
\numberwithin{equation}{section}

\def\R{{\mathbb R}}

\newtheorem{theo}{Theorem}[section]
\newtheorem{lemm}{Lemma}[section]
\newtheorem{coro}{Corollary}[section]

\newtheorem{defi}{Definition}[section]

\def\begeq{\begin{equation}}
\def\endeq{\end{equation}}
\def\p{\partial}

\def\lf{\left}
\def\ri{\right}

\def\R{\Bbb R}

\begin{document}

\title{A property of the Brown-York mass in
 Schwarzschild manifolds}
\author{Xu-Qian Fan$^\dagger $ \&   Kwok-Kun Kwong*}
\thanks{$^\dagger $Research partially supported by the National
Natural Science Foundation of China (10901072, 11291240139) and GDNSF (S2012010010038).\\
*~Research partially supported by Australian Research Council Discovery Grant \#DP0987650}
\address{Department of
Mathematics, Jinan University, Guangzhou,
 510632,
P. R. China.} \email{txqfan@jnu.edu.cn}
\address{School of Mathematical Sciences, Monash University,
Vic 3800, Australia.} \email{Kwok-Kun.Kwong@monash.edu}

\renewcommand{\subjclassname}{%
  \textup{2000} Mathematics Subject Classification}
\subjclass[2000]{Primary 53C20; Secondary 83C99}
\date{Oct. 2012}
\keywords{Brown-York mass, ADM mass, isometric embedding.}

\begin{abstract}
We will extend partially our previous results about the limit of the Brown-York mass of
a family of convex revolution surfaces in the
 Schwarzschild manifold such that these surfaces may have unbounded ratios of their radii.
\end{abstract}
\maketitle\markboth{Xu-Qian Fan $\&$   Kwok-Kun Kwong }{A property
of the Brown-York mass}

\section{Introduction}
In this paper, we will continue to study the limiting behavior of the Brown-York mass of
a family of convex revolution surfaces in the
 Schwarzschild manifold and extend our previous results in \cite{FK}.
Throughout this paper, we will denote $(\mathbb{R}^3,\delta_{ij})$
as the 3-dimensional Euclidean space, $x^1,x^2,x^3$ as the standard
coordinates of $\mathbb{R}^3$, $r$ and $\p$ as the Euclidean distance and the standard
derivative operator on $\R^3$ respectively. For the sake of convenience, let us first recall some
definitions. First of all, we will adopt the following definition of
asymptotically flat manifolds.

\begin{defi}\label{defaf}
A complete three dimensional manifold $(M,\lambda)$ is said to be asymptotically
flat (AF) of order $\tau$ (with one end) if there is a compact
subset $K$ such that $M\setminus K$ is diffeomorphic to
$\R^3\setminus B_R(0)$ for some $R>0$ and in the standard
coordinates in $\R^3$, the metric $\lambda$ satisfies:
\begin{equation} \label{daf1}
\lambda_{ij}=\delta_{ij}+\sigma_{ij}
\end{equation}
with
\begin{equation} \label{daf2}
|\sigma_{ij}|+r|\p \sigma_{ij}|+r^2|\p\p\sigma_{ij}|+r^3|\p\p\p
\sigma_{ij}|=O(r^{-\tau}),
\end{equation}
for some constant $1\ge\tau>\frac{1}{2}$.
\end{defi}
A coordinate system of $M$ near infinity so that the metric tensor
in this system satisfy the above decay conditions is said to be
admissible. In such a coordinate system, we can define the ADM mass
as follows.
\begin{defi}
The Arnowitt-Deser-Misner (ADM) mass (see \cite{ADM}) of an
asymptotically flat manifold $(M,\lambda)$ is defined as:
\begin{equation} \label{defadm1}
m_{ADM}(M,\lambda)=\lim_{r\to\infty}\frac{1}{16\pi}\int_{S_r}
\lf(\lambda_{ij,i}-\lambda_{ii,j}\ri)\nu^jd\Sigma_r^0,
\end{equation}
where $S_r$ is the Euclidean sphere, $d\Sigma_r^0$ is the volume
element induced by the Euclidean metric, $\nu$ is the outward unit
normal of $S_r$ in $\R^3$ and the derivative is the ordinary partial
derivative.
\end{defi}
We always assume that the scalar curvature is in $L^1(M)$ so that
the limit exists in the definition. In \cite{BTK86}, Bartnik showed that the
ADM mass is a geometric invariant.

 On the other hand, there have been many studies on the relation between the
 ADM mass of an AF manifold and the so called quasi-local mass. Let us recall the
 definition of the Brown-York quasi-local mass.
 Suppose
$\lf(\Omega, \mu\ri)$ is a compact three dimensional manifold with
smooth boundary $\p \Omega$, if moreover $\p \Omega$ has positive Gauss
curvature, then the Brown-York mass of $\p\Omega$ is
defined as (see \cite{BY1,BY2}):
\begin{defi}
 \begin{equation} \label{defbym1}
m_{BY}\lf(\p \Omega\ri)=\frac{1}{8\pi}\int_{\p \Omega}(H_0-H)d\sigma
\end{equation}
 where  $H$ is the mean curvature of $\p\Omega$ with respect to the
 outward unit normal and
  the metric $\mu$, $d\sigma$ is the volume element induced on $\p\Omega$ by
 $\mu$ and $H_0$ is the mean curvature  of $\p \Omega$ when
  embedded in  $\R^3$.
\end{defi}

The existence of an isometric embedding in $\R^3$ for $\p\Omega$ was
proved by Nirenberg \cite{Niren}, the uniqueness of  the embedding
was given by \cite{Herglotz43,Sacksteder62,PAV}, so the Brown-York
mass is well-defined.

It can be proved that the Brown-York mass and the Hawking quasi-local mass
\cite{SWH} of
the coordinate spheres tends to the ADM mass in some AF manifolds,
see \cite{BY2,HKGH1,BBWYY,BLP,ST02,FST}, even of nearly round
surfaces \cite{SWW}, and of a family of convex revolution surfaces in an asymptotically
 Schwarzschild manifold \cite{FK} for the Brown-York mass. The ratios of the radii of these surfaces are all bounded. In this paper, we will generalize some results in \cite{FK} partially in that
 we allow the ratios of the radii of the family of surfaces to be unbounded.
We will consider a kind of AF manifolds, called
 Schwarzschild manifolds, defined as follows:
\begin{defi}\label{def: metric}
$(N,g)$ is called a
 Schwarzschild manifold if
   $N=\mathbb{R}^3\setminus K$,  $K$ is a compact set containing the origin, and in the standard coordinates of $\mathbb R^3$,
  \begin{equation*}
 g_{ij}=\phi^4 \delta_{ij},
  \phi=1+\frac{2m}{r},
    m>0.
  \end{equation*}
 \end{defi}
  Clearly, $(N,g)$ is an AF manifold and the scalar curvature of $(N,g)$ is zero \cite{HY} (page 283). Moreover, the ADM mass is equal to $m$.

Let $w(\varphi), z(\varphi)$ be smooth functions on $[0, l]$ such that the surface of revolution generated by $w$ and $z$:
\begin{equation}\label{eq: rev}
(w(\varphi)\cos \theta, w(\varphi)\sin \theta, z(\varphi))
\end{equation}
is a smooth convex surface diffeomorphic to $\mathbb{S}^2$ and
\begin{equation}\label{eq: ass}
\left\{
\begin{array}{ll}
&C_1^2\leq w^2+z^2\leq C_2^2, \textrm{ for }C_1, C_2>0\\
&w'^2+z'^2=1\\&w\geq 0 \textrm{ on } [0,l] \textrm{ and } z(0)> z(l).
\end{array}
\right.
\end{equation}

Let $f(a)$ be a function such that $f(a)\geq 1$ for all $a\geq 1$.
We define the family of surfaces $S_a$ by the parametrization
\begin{equation}
(a w(\varphi) \cos \theta, a w(\varphi)\sin \theta, a h_a(\varphi))
\end{equation}
where $h_a(\varphi)=f(a) z(\varphi)$. Note that $S_a$ forms an exhaustion of $N$ as $a\to \infty$.

We will  prove the following:

\begin{theo}\label{thm1}
Suppose $(S_a, g|_{S_a})$ has positive Gaussian curvature, then the Brown York mass of $S_a$ tends to the ADM mass of $(N, g)$. That is
$$\lim_{a\to\infty}m_{BY}(S_a)
=m.$$
\end{theo}

One example of surfaces satisfying the
conditions in Theorem \ref{thm1}  is the
family of ellipsoids:
   $$S_a=\lf\{ (x^1)^2+(x^2)^2+\frac{2m(x^3)^2}{a}=a^2\ri\}$$
   which has unbounded ratios of their radii as $a\to \infty$.

From Theorem \ref{thm1}, we have
 \begin{coro} \label{co1}
Suppose $\frac{f^2}{a}=o(1)$ for sufficiently large $a$, then
$$\lim_{a\to\infty}m_{BY}(S_a)
=m.$$
 \end{coro}
Clearly the above example shows that the condition  $\frac{f^2}{a}=o(1)$ is not necessary.

This paper is organized as follows. In Section \ref{sec: pf}, we will prove
Theorem \ref{thm1}. Corollary \ref{co1} will be proven in Section 3.

\section*{Acknowledgments}
The authors would like to thank Prof. Luen-Fai Tam for his  constant encouragement and very useful advice. The author
 (XQ) would also like to thank Prof. Youde Wang for his support.

\section{Proof of Theorem \ref{thm1}}\label{sec: pf}

Let us first prove some lemmas.
We can  assume
$w $ is anti-symmetric about $0$ and $l$, $z$
is symmetric about $0$ and
$l $. This implies
\begin{equation}\label{eq: w0}
w (0)=w (l )=z '(0)=z '(l )=0.
\end{equation}

The Gaussian curvature $\overline K $ of \eqref{eq: rev} with respect to $\delta$ is (
\cite{doCarmo} p.162)
$$ \overline K= \frac{z'(w'z''-w''z')}{w} \text{ for } \varphi\in(0,l).$$ So by \eqref{eq: ass}, $z'<0$ on $(0,l)$.

We will sometimes regard $\phi$ as function of $\varphi$ (by restricting on $S_a$) by abuse of notation. We define
 \begin{equation}
 D=\sqrt{w'^2+h'^2}=\sqrt{w'^2+f^2 z'^2}.
 \end{equation}

Similar to Lemma 2.1 in \cite{FK}, we have

 \begin{lemm}\label{lem: 1}
 The functions $\frac{w}{z'}$ and $\frac{w''}{z'}$can be extended continuously to $[0,l]$.
 \end{lemm}

 \begin{proof}
 The Gaussian curvature of the surface given by \eqref{eq: rev} at $(0,0, z(0))$ is $z''(0)^2>0$, so $\lim_{\varphi\to 0}\frac{w}{z'}= \frac{w'(0)}{z''(0)}<\infty$. The case for $\varphi=l$ is the same.

 The plane curve $(w(\varphi), z(\varphi))$ has curvature $k=w''z'- z'' w'= (w',z')\cdot (-z'',w'')$.
 Since $(w'',z'')\perp(w',z')$ and $(w'',z'')\perp(-z'',w'')$, we have $w''=kz'$. From this we see that $\frac{w''}{z'}$ can be extended
 to $k(0)$ at $\varphi=0$. The case for $\varphi=l$ is the same.
 \end{proof}

\begin{lemm}\label{lem: est} The following functions can be extended continuously to $[0,l]$ such that
\begin{equation}\label{eq: est}
\begin{split}
\frac{w}{h'}=O\lf(\frac{1}{f}\ri),\frac{w''}{h'}=O\lf(\frac{1}{f}\ri), \\
D'=O\lf(\frac{f^2}{D}\ri), \frac{D'}{h'}=O\lf(\frac{f}{D}\ri)
\end{split}
\end{equation}
\begin{equation}\label{eq: est2}
\begin{split}
\phi'=O\lf(\frac{f a}{r^2}\ri), \frac{\phi'}{h'}=O\lf(\frac{f a^2}{r^3}\ri),
\phi''=O\lf(\frac{f^2 a^2}{r^3}\ri).%, \frac{\phi' p}{h'^2}=O(\frac{a^2}{r^3}),\\
\end{split}
\end{equation}
\end{lemm}

\begin{proof}
The first line of \eqref{eq: est} follows from Lemma \ref{lem: 1}. As $D'= \frac{w'w''+f^2z'z''}{D}$, $|D'|\leq \frac{Cf^2}{D}$. Also, using the fact that $\frac{w}{h'}=O(f^{-1})=O(1)$,
$|\frac{D'}{h'}|= |\frac{1}{D}(w'\frac{w''}{h'}+fz'')|= O(\frac{f}{D}) $.

\noindent
For \eqref{eq: est2}, we have $\phi'=-\frac{ma^2}{2r^3}(ww'+f^2zz')$. So
\begin{equation*}
\begin{split}
|\phi'|\leq \frac{ma^2}{2r^3} (w^2 + f^2 z^2)^{\frac{1}{2}} ( w'^2 + f^2 z'^2)^{\frac{1}{2}}
&=\frac{mar}{2r^3}( w'^2 + f^2 z'^2)^{\frac{1}{2}}\\
&= O\lf(\frac{fa}{r^2}\ri).
\end{split}
\end{equation*}
Using \eqref{eq: est} again,
$$\left|\frac{\phi'}{h'}\right|=-\frac{ma^2}{2r^3}\lf(w'\frac{w}{h'}+fz\ri)= O\lf(\frac{fa^2}{r^3}\ri).
$$
Finally,
\begin{equation*}
\begin{split}
|\phi''|
=&\left| \frac{3m}{2} r^{-5} a^4 (ww'+hh')^2- \frac{m}{2} r^{-3} a^2( (w')^2+(h')^2+ww''+hh'')\right|\\
\leq&  \frac{3m}{2} r^{-5} a^4 (w^2+h^2)((w')^2+(h')^2)\\
& +\frac{m}{2} r^{-3} a^2( ((w')^2+(h')^2)+(w^2+h^2)^{\frac{1}{2}}((w'')^2+(h'')^2)^{\frac{1}{2}})\\
=& \frac{3m}{2} r^{-3} a^2 ((w')^2+(h')^2)+ O(r^{-3}a^2f^2)=O(r^{-3}a^2f^2).
\end{split}
\end{equation*}
Hence Lemma \ref{lem: est} holds.
\end{proof}

We want to compute the mean curvatures. By Lemma 2.4 in \cite{FK}, we have
\begin{lemm}\label{lem: bar H}
For a smooth revolution surface in $(\mathbb{R}^3,\delta)$
parametrized by
\begin{equation}
(au(\varphi)\cos\theta, au(\varphi)\sin\theta, av(\varphi)), \quad
0<\varphi<l, 0<\theta< 2\pi,
\end{equation}
its mean curvature $\overline{H}$ with respect to $\delta$ is
\begin{equation}
\overline H=\frac{u''}{aTv'}-\frac{T'u'}{aT^2v'}-\frac{v'}{aTu}\quad\text{where
$T=\sqrt{u'^2+v'^2}$.}
\end{equation}
\end{lemm}

Similar to Lemma 2.5 in \cite{FK}, we can get
\begin{lemm}\label{lem: H}
The mean curvature $H$ of $S_a$ with respect to $g$ is
$$H=
\frac{w''}{a\phi^2Dh'}
-\frac{D'w'}{a\phi^2D^2h'}
-\frac{h'}{a\phi^2Dw}+4\phi^{-3}n(\phi)
 $$where $n$ is the outward unit normal vector of $S_a$ with respect to
$\delta$.
\end{lemm}

\begin{proof}
It is similar to the proof of Lemma 2.5 in \cite{FK}. For completeness, we sketch it here. By Lemma \ref{lem: H}, the mean curvature of $S_a$ with respect
to $\delta$ is
\begin{equation}\label{eq:lmsc1}
\overline H=\frac{w''}{aDh'}-\frac{D'w'}{aD^2h'}-\frac{h'}{aDw}.
\end{equation}
The mean curvature $ {H}$ of $S_a$ with respect to $ {g}$ is
(\cite{SYL79}, page 72):
\begin{equation}\label{eq:lmsc2}
 {H}= {\phi}^{-2}\lf(\bar{H}+4 {\phi}^{-1}n\lf( {\phi}\ri)\ri)
\end{equation}
where $n$ is the outward unit normal vector of $S_a$ with respect to $\delta$. Submitting \eqref{eq:lmsc1} to \eqref{eq:lmsc2}, we can get Lemma  \ref{lem: H}.
\end{proof}

\begin{lemm}\label{lem: uv}
Suppose $(S_a, g|_{S_a})$ has positive Gaussian curvature such that it can be uniquely isometrically embedded into $\mathbb{R}^3$ (for sufficiently large $a$), then the embedding is given by
\begin{equation}\label{eq: embed}
(au(\varphi)\cos\theta,au(\varphi)\sin\theta,av(\varphi)), \varphi\in [0, l], \theta\in [0, 2\pi]
\end{equation}
where
\begin{equation}\label{eq:va1}
u=\phi^2 w\text{\quad and \quad}
v'=\phi^2 h'
\lf(1-\frac{4\phi'ww'}{\phi h'^2}-\frac{4\phi'^2w^2}{\phi^2
h'^2}\ri)^{\frac{1}{2}} .
\end{equation}
\end{lemm}

\begin{proof}
The proof is similar to that of Lemma 2.6 in \cite{FK}. In $(\varphi, \theta)$ coordinates, the metric on $S_a$ induced by
$g$ can be written as:
\begin{equation}\label{eq: metric}
ds^2=a^2\phi^4(w'^2+h'^2)d\varphi^2+a^2\phi^4 w^2 d\theta^2.
\end{equation}
We can regard $(S_a,ds^2)$ as the sphere with the
metric $ds^2$. We want to find two functions $u,v$ such that the
surface written as  the form \eqref{eq: embed} is an embedded
surface  $S^e_a$ of $S_a$ into $(\mathbb{R}^3,\delta)$. First of
all, the induced metric by the Euclidean metric on the surface which
is of the form \eqref{eq: embed}  can be written as:
$$ds^2_e=a^2\lf(u'^2+v'^2\ri)d\varphi^2+a^2u^2d\theta^2 .$$
Comparing this with \eqref{eq: metric}, one can choose
\begin{equation}\label{eq: u=phi^2w}
u=\phi^2w \textrm{ and } u'^2+v'^2=\phi^4D^2.
\end{equation}
Consider
\begin{equation}\label{eq: phi^4-u'^2}
\begin{split}
\phi^4(w'^2+h'^2)-u'^2
&=\phi^2 (\phi^2(w'^2+h'^2)-(2\phi'w+\phi w')^2)\\
%&=\phi^2 (\phi^2 (w'^2+h'^2)-(2\phi'w+\phi w')^2)\\
&=\phi^2 (\phi^2h'^2-4\phi\phi'ww'-4\phi'^2w^2)\\
&=\phi^4h'^2 \lf(1-\frac{4\phi'ww'}{\phi
h'^2}-\frac{4\phi'^2w^2}{\phi^2 h'^2}\ri) .
\end{split}
\end{equation}
By Lemma \ref{lem: est}, the functions
$\displaystyle\frac{\phi'ww'}{\phi h'^2}, \frac{\phi'^2w^2}{\phi^2
h'^2}$ can be extended continuously on $[0, l]$ with
$\displaystyle\frac{\phi'ww'}{\phi h'^2}=O(a^{-1}),
\frac{\phi'^2w^2}{\phi^2 h'^2}=O(a^{-2})$. So $\displaystyle
1-\frac{4\phi'ww'}{\phi h'^2}-\frac{4\phi'^2w^2}{\phi^2 h'^2}> 0$
for sufficiently large $a$. For these $a$, we can take
$$
v'=\phi^2 h' \lf(1-\frac{4\phi'ww'}{\phi
h'^2}-\frac{4\phi'^2w^2}{\phi^2 h'^2}\ri)^{\frac{1}{2}} ,
$$
so that $
u'^2+v'^2=\phi^4(w'^2+h'^2)
$.
Note that $v'$ is an odd function for
$\varphi\in[-l,l].$ By choosing an initial value, one can get an
even function $v$. By the above argument, one has
$$%\label{eq: v'}
v'=\phi^2h' \lf(1-\frac{2\phi'ww'}{ h'^2}+O\lf(a^{-2}\ri)\ri).
$$

From \eqref{eq: u=phi^2w} and \eqref{eq: phi^4-u'^2}, near
$\varphi=0$, $u,\ v$ can be extended naturally to
$(-\varepsilon,\varepsilon)$ for some $\varepsilon>0$. Since $u$ is
an odd function in $\varphi$ , $v$ is an even function in $\varphi$,
and $u'^2+v'^2=T^2>0$, the generating curve in $\{x^2=0\}$ is
symmetric with respect to $x^3$-axis, and is smooth at $\varphi=0$.
Similarly, it is also smooth at $\varphi=l$. Hence the revolution
surface determined by the choice of $u,\ v$ as above can be
extended smoothly to a closed revolution surface, which is the embedded surface of $S_a$.
\end{proof}

Now we are ready to prove Theorem \ref{thm1}.
\noindent
\begin{proof}[Proof of Theorem \ref{thm1}]
Let $F= \phi^2 D$, then by Lemma \ref{lem: bar H} and Lemma \ref{lem: H},
\begin{equation}\label{eq: H0-H}
\begin{split}
H_0-H&=\lf(\frac{u''}{aFv'}-\frac{w''}{a\phi^2Dh'}\ri)
-\lf(\frac{F'u'}{aF^2v'}-\frac{D'w'}{a\phi^2D^2h'}\ri)
\\
&\quad
-\lf(\frac{v'}{aFu}-\frac{h'}{a\phi^2Dw}\ri)-4\phi^{-3}n(\phi).
\end{split}
\end{equation}

\noindent
By Lemma \ref{lem: est}, we have
\begin{equation}\label{eq: I}
\begin{split}
\frac{u''}{aFv'}-\frac{w''}{a\phi^2Dh'}
 =\frac{4\phi'
w'}{aDh'}+\frac{2\phi''w}{aDh'} +\frac{2\phi'ww'w''}{aDh'^3}
+O\lf(\frac{f a^3}{D r^6 }\ri),
\end{split}
\end{equation}

\begin{equation}\label{eq: II}
\begin{split}
-\frac{F'u'}{aF^2v'}+\frac{D'w'}{a\phi^2D^2h'}
=&-\frac{2\phi'D'w}{aD^2
h'}-\frac{2\phi'w'}{aDh'}-\frac{2\phi'ww'^2D'}{aD^2h'^3}\\
&+O\lf(\frac{f^2 a^2 w}{D^3 r^5}\ri)+O\lf(\frac{f a^3}{D r^6}\ri),
\end{split}
\end{equation}

\noindent
and
\begin{equation}\label{eq: III}
\begin{split}
-\frac{v'}{aFu}+\frac{h'}{a\phi^2Dw}=\frac{2\phi'w'}{aDh'}+O\lf(\frac{f a^3}{D r^6}\ri).
\end{split}
\end{equation}
Summing \eqref{eq: I}, \eqref{eq: II} and \eqref{eq: III}, comparing
with \eqref{eq: H0-H}, we have
\begin{equation}\label{eq:msum}
\begin{split}
&H_0-H\\
&=\lf(\frac{4\phi' w'}{aDh'}+\frac{2\phi''w}{aDh'}-\frac{2\phi'wh''}{a D
h'^2}\ri)
-4\phi^{-3}n(\phi)+O\lf(\frac{f^2 a^2 w}{D^3 r^5}\ri)+O\lf(\frac{f a^3}{D r^6}\ri).
\end{split}
\end{equation}

Note that by Lemma \ref{lem: est},
\begin{equation*}
H_0-H=O\lf(\frac{f a}{D r^3}\ri)+O\lf(\frac{f^2 a^2 w}{D^3 r^5}\ri).
\end{equation*}

We claim that
\begin{equation}\label{eq: sigma}
\lim_{a\rightarrow \infty}\int_{S_a}(H_0-H)d\sigma=\lim_{a\rightarrow \infty}\int_{S_a}(H_0-H)d\sigma_0.
\end{equation}
Noting that $d\sigma-d\sigma_0=O(r^{-1})d\sigma_0$, it suffices to show that
\begin{equation}\label{eq: lim1}
\lim_{a\rightarrow \infty}\int_{S_a}O\lf(\frac{f a}{D r^3}\ri)O(r^{-1})d\sigma_0=O(a^{-1})
\end{equation}
and
\begin{equation}\label{eq: lim2}
\int_{S_a}O\lf(\frac{f^2 a^2 w}{D^3 r^5}\ri)O(r^{-1})d\sigma_0=O(a^{-1}),
\end{equation}
which in turn is implied by the stronger result
\begin{equation}\label{eq: lim3}
\int_{S_a}O\lf(\frac{f^2 a^2 w}{D^3 r^5}\ri)d\sigma_0=O(a^{-1}).
\end{equation}

To prove \eqref{eq: lim1}. Since $d\sigma_0=a^2 D w d\varphi d\theta$,  let $f^2=1+\alpha^2$, consider
\begin{equation*}
\begin{split}
0\leq \int_{S_a}\frac{f a}{D r^4}d\sigma_0
&=\frac{2\pi}{a} \int_0^l \frac{f w}{(w^2+f^2 z^2)^2}d\varphi \\
&\leq \frac{2\pi }{a} \int_0^l \frac{(1+\alpha)w }{(w^2+ z^2+\alpha^2 z^2)^2}d\varphi \\
&\leq \frac{2\pi }{a} \left(l+\int_0^l \frac{\alpha w}{\lf(C_1^2+\alpha^2 z^2\ri)^2}d\varphi \right)\\
&= \frac{2\pi }{a} \left(l+\frac{1}{C_1^4}\int_{s(0)}^{s(l)} \frac{\alpha w/z'}{\lf(1+\frac{\alpha^2 z^2}{C_1^2}\ri)^2}d s \right)\\
&\leq \frac{2\pi }{a} \left(l+\frac{C}{C_1^4}\int_{s(0)}^{s(l)} \frac{\alpha }{\lf(1+\frac{\alpha^2 z^2}{C_1^2}\ri)^2}d s \right)\\
&\leq \frac{2\pi }{a} \left(l+\frac{C}{C_1^3}\int_{y_1}^{y_2} \frac{1 }{(1+y^2)^2}d y \right)\\
&\leq \frac{2\pi }{a} \left(l+\frac{C\pi}{C_1^3}\right)
\end{split}
\end{equation*}
where we have used the fact that $\displaystyle \frac{w}{z'}$ can be extended to a continuous function on $[0, l]$ which is bounded by $C$. For \eqref{eq: lim3}, using similar computations, we have
\begin{equation*}
\begin{split}
0\leq \int_{S_a}\frac{f^2 a^2 w}{D^3 r^5}d\sigma_0
&\leq \frac{C_3}{a}+ \frac{C_3}{a}\int_0^l \frac{\alpha^2 w^2 }{(1+\alpha^2 z'^2)(1+ \frac{\alpha^2 z^2}{C_1^2})^{5/2}}d\varphi \\
&\leq \frac{C_3}{a}+ \frac{C_3}{a}\int_0^l \frac{\alpha^2 w^2 }{(2\alpha z')(1+ \frac{\alpha^2 z^2}{C_1^2})^{5/2}} d\varphi\\
&\leq \frac{C_3}{a}+ \frac{C_3 C}{2a}\int_0^l \frac{\alpha w }{(1+ \frac{\alpha^2 z^2}{C_1^2})^{5/2}} d\varphi\\
&\leq \frac{C_4}{a}.
\end{split}
\end{equation*}
Hence \eqref{eq: sigma} is true:
\begin{equation*}
\lim_{a\rightarrow \infty}\int_{S_a}(H_0-H)d\sigma=\lim_{a\rightarrow \infty}\int_{S_a}(H_0-H)d\sigma_0.
\end{equation*}
Next, by \eqref{eq:msum}
\begin{equation*}
\begin{split}
&H_0-H\\
&=\lf(\frac{4\phi' w'}{aDh'}+\frac{2\phi''w}{aDh'}-\frac{2\phi'wh''}{a D
h'^2}\ri)
-4\phi^{-3}n(\phi)+O\lf(\frac{f^2 a^2 w}{D^3 r^5}\ri)+O\lf(\frac{f a^3}{D r^6}\ri).
\end{split}
\end{equation*}
Consider
\begin{equation} \label{eq:macth}
\begin{split}
&\int_{S_a} \lf(\frac{4\phi' w'}{aDh'}+\frac{2\phi''w}{aDh'}
-\frac{2\phi'wh''}{a D h'^2}\ri) d\sigma_0\\
 &=2\pi a\int_{0}^l
\lf(\frac{4\phi' ww'}{h'}+\frac{2\phi''w^2}{h'}
-\frac{2\phi'w^2h''}{ h'^2}\ri) d\varphi\\
&=2\pi a\int_{0}^l  \frac{d}{d\varphi}\lf(\frac{2\phi' w^2}{h'}\ri)d\varphi\\
&=0.
\end{split}
\end{equation}
So by \eqref{eq: lim1} and \eqref{eq: lim2}, we have
\begin{equation}\label{eq:byl1}
\begin{split}
\frac{1}{8\pi}\int_{S_a}(H_0-H)d\sigma_0
&=-\frac{1}{2\pi}\int_{S_a} \phi^{-3}n(\phi)d\sigma_0+O\lf(a^{-1}\ri)\\
&=\frac{1}{4\pi}\int_{S_a} n(\phi^{-2})d\sigma_0+O\lf(a^{-1}\ri).
\end{split}
\end{equation}

For each $a$, choose $\partial B_a$ which is a Euclidean coordinate sphere enclosing $S_a$ and let $\Omega_a$ be the region between $\partial B_a$ and $S_a$.
The ADM mass of $N$ is defined as
\begin{equation*}
m_{ADM}=\lim_{a\to \infty}\frac{1}{16\pi} \int_{\partial B_a}
(g_{ij,i}-g_{ii,j})n^j d\sigma_0=-\lim_{a\to \infty}\frac{1}{2\pi} \int_{\partial B_a}
\phi^3 n(\phi) d\sigma_0
\end{equation*}
where $n$ is the unit outward normal of $\partial B_a$ with respect to
$\delta$. As $n(\phi)=O\lf(r^{-2}\ri)$,
\begin{equation*}
\phi^3 n(\phi) = n(\phi)+O(r^{-3}).
\end{equation*}
Clearly
\begin{equation*}
\lim_{a\to \infty} \int_{\partial B_a}
O(r^{-3}) d\sigma_0=0,
\end{equation*}
so
\begin{equation}\label{eq: adm}
m_{ADM}=-\lim_{a\to \infty}\frac{1}{2\pi} \int_{\partial B_a}
 n(\phi) d\sigma_0.
\end{equation}
By divergence theorem and the fact that $\Delta \phi=0$,
\begin{equation*}
\begin{split}
\frac{1}{4\pi}\int_{S_a} n(\phi^{-2})d\sigma_0
&=\frac{1}{4\pi}\int_{\partial B_a} n(\phi^{-2})d\sigma_0-\frac{1}{4\pi}\int_{\Omega_a}\Delta (\phi^{-2})dV_0\\
&=\frac{1}{4\pi}\int_{\partial B_a} n(\phi^{-2})d\sigma_0-\frac{1}{4\pi}\int_{\Omega_a}6|\nabla \phi|^2 \\
&=\frac{1}{4\pi}\int_{\partial B_a} n(\phi^{-2})d\sigma_0+\int_{\Omega_a}O(r^{-4})dV_0\\
&=-\frac{1}{2\pi} \int_{\partial B_a}
 n(\phi) d\sigma_0+\int_{\partial B_a}
O(r^{-3}) d\sigma_0+O(a^{-1})\\
&=-\frac{1}{2\pi} \int_{\partial B_a}
 n(\phi) d\sigma_0+O(a^{-1}).
 \end{split}
\end{equation*}
So by \eqref{eq:byl1} and \eqref{eq: adm}, we have
\begin{equation*}
\lim_{a\rightarrow \infty}\frac{1}{8\pi}\int_{S_a}(H_0-H)d\sigma_0=m_{ADM}.
\end{equation*}
Since
\begin{equation*}
\begin{split}
\lim_{a\to \infty} \frac{1}{8\pi} \int_{S_a}(H_0-H)d\sigma
=\lim_{a\to \infty} \frac{1}{8\pi} \int_{S_a}(H_0-H)d\sigma_0,
\end{split}
\end{equation*}
we can conclude that
\begin{equation*}
\lim_{a\to\infty}m_{BY}(S_a)=m_{ADM}(N,g).
\end{equation*}
This completes the proof of our result.
\end{proof}

\section{Proof of Corollary \ref{co1}}
First of all, we have the following:
\begin{lemm}\label{lem: bar K}
The Gaussian curvature $K(\delta)$ of $S_a$ with metric induced by $\delta $ is positive.
\end{lemm}

\begin{proof}
Let $d\overline s^2$ be the metric on $S_a$ induced by $\delta$.
The Gaussian curvature of the revolution surface given by \eqref{eq: rev} is $K_0=-\frac{w''}{w}>0$.
% (cf. DoCarmo p.162 for example)
 On the other hand,
\begin{equation}\label{msiem}
 d\overline s^2 =  a^2((w'^2 + f^2 z'^2) d\varphi^2+ w^2 d\theta^2) = \overline Ed\varphi^2+ \overline G d\theta^2.
\end{equation}
The Gaussian curvature of $d\overline s^2$ is then given by
\begin{equation} \label{gcex1}
\begin{split}
K(\delta)
&=-\frac{1}{2\sqrt{\overline E\overline G}}\left(\left(\frac{\overline  E_\theta}{ \sqrt{\overline  E \overline G}}\right)_\theta+ \left(\frac{\overline  G_\varphi}{ \sqrt{\overline  E\overline  G}}\right)_\varphi\right)\\
&=-\frac{1}{a^2w\sqrt{w'^2+ h'^2}}\left(\frac{w'}{ \sqrt{w'^2+f^2z'^2}}\right)'\\
&=-\frac{1}{a^2w\sqrt{w'^2+ h'^2}}\left(\frac{w''}{ \sqrt{w'^2+f^2z'^2}}-\frac{w'(w'w''+f^2z'z'')}{ (w'^2+f^2z'^2)^{\frac{3}{2}}}\right)\\
&=-\frac{1}{a^2w\sqrt{w'^2+ h'^2}}\left(\frac{w''}{ \sqrt{w'^2+f^2z'^2}}-\frac{w'(w'w''-f^2w'w'')}
{ (w'^2+f^2z'^2)^{\frac{3}{2}}}\right)\\
&=\frac{-w''f^2}{a^2w(w'^2+f^2z'^2)^2}\\
&=\frac{K_0f^2}{a^2D^4}>0,
\end{split}
\end{equation}
where we have used $w'w''+z'z''=0$.
\end{proof}

\begin{proof}[Proof of Corollary \ref{co1}]
By Theorem \ref{thm1}, we just need to show that $g|_{S_a}$ has positive Gaussian curvature as $a>>1.$ By abuse of notations, we denote $\delta|_{S_a}$ simply by $\delta $ and $g|_{S_a}$ by $g$.  Noting that $g=\phi^4 \delta$. Similar to (2.4) in \cite{K} or (2.14) in \cite{GS}, we have
$$ K(g)= \phi^{-4}(K(\delta)-4\Delta_
S (\log \phi))$$
where $\Delta_S $ is the Laplacian on $(S_a, \delta|_{S_a})$.
We have the following formula:
$$\Delta_S \psi=\Delta _{\mathbb{R}^3}\psi-\nabla_{\mathbb{R}^3} ^2 \psi(n,n)-\overline H n(\psi) $$
where $\overline H$ and $n$ are the mean curvature and the unit outward normal vector of $S_a$ with respect to $\delta$ respectively. Letting $\psi=\log \phi$, we have
$$K(g)= \phi^{-4}(K(\delta)-4\Delta _{\mathbb{R}^3}\psi+4\nabla_{\mathbb{R}^3} ^2 \psi(n,n)+4\overline H n(\psi)).$$
Since $\Delta _{\mathbb{R}^3}\phi=0$, we have
$$\Delta_{\mathbb{R}^3}\psi= \frac{\Delta_{\mathbb{R}^3} \phi}{\phi}-\frac{|\nabla \phi|^2}{\phi^2}=-\frac{|\nabla \phi|^2}{\phi^2}\leq0.$$

 By direct calculations, in Euclidean coordinates, for $\psi=\psi(r)$, we have
$$ (\nabla_{\mathbb{R}^3}^2\psi)_{ij}= \psi'' \frac{x^ix^j}{r^2} + \psi' \frac{\delta_{ij}}{r}- \psi' \frac{x^ix^j}{r^3}.$$
Therefore
$$\nabla^3_{\mathbb{R}^3} \psi (n,n)= \frac{\psi''}{r^2} \langle X, n\rangle ^2+\frac{\psi'}{r} -\frac{ \psi'}{r^3}\langle X, n\rangle ^2$$
where $X$ is the position vector. For $\psi=\log \phi=\log (1+\frac{m}{2r})$, it is easy to see that the negative part of $\nabla ^2 _{\mathbb{R}^3} \psi(n,n)$ is of the order $O(\frac{1}{r^3})$:
$$(\nabla ^2 _{\mathbb{R}^3} \psi(n,n))_-=O\left(\frac{1}{r^3}\right).$$
By Lemma \ref{lem: bar H}, $\overline H= O\left(\frac{f}{aD}\right)$, thus $$\overline H n(\psi)=O\left(\frac{f}{ar^2D}\right).$$
By \eqref{eq: ass}, we have $$ \frac{1}{r^3}\leq \frac{f}{C_1ar^2D}.$$
We conclude that the negative part of $-4 \Delta _{S}(\log \phi)$ is of order $ O(\frac{f}{ar^2D})$.
Lemma \ref{lem: bar K} shows that $K(\delta)$ is of order $O(\frac{f^2}{ a^2D^4})$. As $f\leq D$ and $a\leq r$, it is easy to see that $K(\delta)$ dominates the negative part of $-4 \Delta _{S}(\log \phi)$ if $\frac{f^2}{a}=o(1)$. Hence the Gaussian curvature of $S_a$ is positive for $a$ large enough.
\end{proof}

\end{document}